\documentclass[11pt,english,reqno]{amsart}

\usepackage{amsmath,amssymb,enumerate}

\usepackage{mathpazo}
\usepackage{mathrsfs}
\DeclareFontFamily{OMS}{rsfs}{\skewchar\font'60}
\DeclareFontShape{OMS}{rsfs}{m}{n}{<-5>rsfs5 <5-7>rsfs7 <7->rsfs10 }{}
\DeclareSymbolFont{rsfs}{OMS}{rsfs}{m}{n}
\DeclareSymbolFontAlphabet{\scr}{rsfs}
\DeclareSymbolFontAlphabet{\scr}{rsfs}
\usepackage[all]{xy}

\usepackage{array}

\usepackage{rotating}

\usepackage{tikz}
\usepackage{babel}
\usepackage{amstext}
\usepackage{amsmath}
\usepackage{amsfonts}

\usepackage{latexsym}
\usepackage{ifthen}
\usepackage{setspace}

\usepackage[colorlinks,linkcolor=red,anchorcolor=green,citecolor=blue]{hyperref}
\hypersetup{linktocpage = true}

\usepackage{enumitem}

\newcommand\cE{{\mathcal E}}
\newcommand\cF{{\mathcal F}}

\newcommand\cO{{\mathcal O}}

\newcommand\cV{{\mathcal V}}

\newcommand\bbC{{\mathbb C}}

\newcommand\bbP{{\mathbb P}}
\newcommand\bbQ{{\mathbb Q}}

\newcommand\bbZ{{\mathbb Z}}

\newcommand{\rk}{{\rm rk}}

\DeclareMathOperator*{\bs}{Bs}

\DeclareMathOperator*{\pic}{Pic}

\newcommand{\Chow}[1]{\ensuremath{\mbox{\rm Chow}(#1)}}

\newcommand{\cHom}[2]{\ensuremath{\mathcal{H}om_{\mathcal{O}_X}(#1,#2)}}

\newcommand{\RC}[1]{\ensuremath{\mbox{\rm RatCurves}^n(#1)}}

\newtheorem{lemma1}{}[section]
\newenvironment{lemma}{\begin{lemma1}{\bf Lemma.}}{\end{lemma1}}

\newenvironment{thm}{\begin{lemma1}{\bf Theorem.}}{\end{lemma1}}
\newenvironment{prop}{\begin{lemma1}{\bf Proposition.}}{\end{lemma1}}
\newenvironment{cor}{\begin{lemma1}{\bf Corollary.}}{\end{lemma1}}
\newenvironment{remark}{\begin{lemma1}{\bf Remark.}\rm}{\end{lemma1}}

\newenvironment{defn}{\begin{lemma1}{\bf Definition.}}{\end{lemma1}}

\newenvironment{question}{\begin{lemma1}{\bf Question.}}{\end{lemma1}}

\newenvironment{thm A}{{\bf Theorem A.}}{}
\newenvironment{thm B}{{\bf Theorem B.}}{}
\newenvironment{thm C}{{\bf Theorem C.}}{}
\newenvironment{thm D}{{\bf Theorem D.}}{}
\newenvironment{remark*}{{\bf Remark.}}{}
\newenvironment{example*}{{\bf Example.}}{}
\newenvironment{assumption*}{{\bf Assumption.}}{}

\newenvironment{thm M}{{\bf Macaulay's Theorem.}}{}

\setlist[itemize]{leftmargin=*}
\setlist[enumerate]{leftmargin=*,align=left,nolistsep}

\numberwithin{equation}{section}

\makeatletter

\ifnum\@ptsize=0  \fi \ifnum\@ptsize=2  \fi 

\title{Second Chern class of Fano manifolds and anti-canonical geometry} 
\date{\today}


\author{Jie Liu}

\address{Jie Liu, Universit{\'e} de C{\^o}te d'Azur, CNRS, LJAD, France}
\email{jliu@unice.fr}

\begin{document}

\begin{abstract}
	Let $X$ be a Fano manifold of Picard number one. We establish a lower bound for the second Chern class of $X$ in terms of its index and degree. As an application, if $Y$ is a $n$-dimensional Fano manifold with $-K_Y=(n-3)H$ for some ample divisor $H$, we prove that $h^0(Y,H)\geq n-2$. Moreover, we show that the rational map defined by $\vert mH\vert$ is birational for $m\geq 5$, and the linear system $\vert mH\vert$ is basepoint free for $m\geq 7$. As a by-product, the pluri-anti-canonical systems of singular weak Fano varieties of dimension at most $4$ are also investigated.
\end{abstract}

\maketitle

\tableofcontents

\vspace{-0.2cm}
\section{Introduction}
An important outstanding problem in differential geometry asks which Fano manifolds $X$ with $\rho(X)=1$ admit a K{\"a}hler-Einstein metric, known as the so-called Yau-Tian-Donaldson conjecture. This conjecture was solved by Chen-Donaldson-Sun and Tian in \cite{ChenDonaldsonSun2014,ChenDonaldsonSun2015,Tian2015}. A weaker and more algebraic related question would ask whether the tangent bundle $T_X$ is (semi-)stable with respect to the anti-canonical divisor $-K_X$. We denote by $i_X$ the index of $X$ and set $\dim(X)=n$. By the work of Hwang, Peternell-Wi{\'s}niewski and Reid, there is a positive answer if $n\leq 6$, or $i_X=1$, or $i_X>(n+1)/2$ \cite{Reid1977,PeternellWisniewski1995,Hwang1998,Hwang2001}. Moreover, if $T_X$ is (semi-)stable, then the famous Bogomolov inequality gives a lower bound for the second Chern class of $X$ 
\begin{equation}\label{PositivityChern}
c_2(X)\cdot H^{n-2}\geq \frac{n-1}{2n}c^2_1(X)\cdot H^{n-2}=\frac{(n-1)i_X^2}{2n}H^n,
\end{equation}
where $H$ is the ample generator of $\pic{(X)}$. Our first result is the following weaker version of inequality \eqref{PositivityChern} without assuming the (semi-)stability of $T_X$. We will focus on the case $n\geq 7$ and $i_X\geq 2$ since the inequality \eqref{PositivityChern} holds for $n\leq 6$ or $i_X=1$ (cf. \cite{Reid1977,Hwang1998}).

\begin{thm}\label{Chern class}
	Let $X$ be a $n$-dimensional Fano manifold with $\rho(X)=1$ such that $n\geq 7$. Let $H$ be the fundamental divisor of $X$ and let $i_X$ be the index of $X$. 
	\begin{enumerate}
		\item\label{i_X=2} If $i_X=2$, then 
		\[c_2(X)\cdot H^{n-2}\geq \frac{11n-16}{6n-6} H^n.\]
		
		\item\label{i_X>2} If $3\leq i_X\leq n$, then 
		\[c_2(X)\cdot H^{n-2}\geq \frac{i_X(i_X-1)}{2} H^n.\]
	\end{enumerate}
\end{thm} 

It is easy to check that the inequalities given in Theorem \ref{Chern class} are strictly weaker than the expected inequality \eqref{PositivityChern} except $i_X=n$, so the statement is only interesting for $2\leq i_X\leq (n+1)/2$ and $n\geq 7$. Moreover, recall that $i_X>n$ if and only if $X\cong\bbP^n$ and $i_X=n+1$, and the stability of $T_{\bbP^n}$ is well-known. As the first application of this theorem, we prove an effective non-vanishing theorem for Fano manifolds of dimension $n$ and index $n-3$. 

\begin{thm}\label{Non-vanishingH}
	Let $X$ be a Fano manifold of dimension $n\geq 4$ and index $n-3$. Let $H$ be the fundamental divisor. Then $h^0(X,H)\geq n-2$.
\end{thm}

This result was proved for $n\not=6$, $7$ by Floris in \cite[Theorem 1.2]{Floris2013}. Moreover, combining \cite[Theorem 1.6]{HoeringVoisin2011} with Theorem \ref{Non-vanishingH} yields the following theorem.

\begin{cor}
	Let $X$ be a $n$-dimensional Fano manifold with index $n-3$. Then the group $H^{2n-2}(X,\bbZ)$ is generated over $\bbZ$ by classes of curves.
\end{cor}

As a second application of our theorem, we consider pluri-anti-canonical systems. If $X$ is a weak $\bbQ$-Fano threefold with at worst canonical singularities, then there exists a constant $m_3$ such that the pluri-anti-canonical map $\Phi_{-m}$ is birational for $m\geq m_3$ (cf.~\cite[Corollary 1.3]{KollarMiyaokaMoriTakagi2000}). In higher dimension, the existence of such a constant $m_n$ was proven by Birkar very recently under some mild assumption on the singularities (cf.~\cite[Theorem 1.2]{Birkar2016a}). On the other hand, by effective Basepoint Free Theorem \cite[Theorem 1.1]{Kollar1993}, the linear system $\vert-mK_X\vert$ is basepoint free for $m=2(n+2)!(n+1)r$ if $X$ has only klt singularities and $r$ is the Gorenstein index of $X$. This leads us to ask the following two natural questions.

\begin{question}\label{Question}
	Let $X$ be a weak Fano variety with at most canonical Gorenstein singularities.
	\begin{enumerate}		
		\item\label{Qfreeness} Find the optimal constant $f(n)$ depending only on $\dim(X)=n$ such that the linear system $\vert -mK_X\vert$ is basepoint free for all $m\geq f(n)$.
		
		\item\label{QBirationality} Find the optimal constant $b(n)$ depending only on $\dim(X)=n$ such that the rational map $\Phi_{-m}$ corresponding to $\vert-mK_X\vert$ is a birational map for all $m\geq b(n)$.
	\end{enumerate}
\end{question}

Question (1) is closely related to Fujita's basepoint freeness conjecture, and Question (2) has attracted a lot of interest over the past few decades \cite{Ando1987,Fukuda1991,Chen2011,ChenJiang2016}, etc. On the other hand, we remark that the constants $f(n)$ and $b(n)$ in Question \ref{Question} are invariant if we replace weak Fano varities by Fano varieties (cf. Proposition \ref{WeakFano to Fano}). In dimension $2$, we have the following known result.

\begin{thm}\cite{Ando1987,Reider1988}\label{Two}
	Let $S$ be a projective surface with at most canonical singularities. If the anti-canonical divisor $-K_S$ is nef and big, then 
	\begin{enumerate}		
		\item the linear system $\vert -mK_S\vert$ is basepoint free for all $m\geq 2$;
		
		\item the morphism $\Phi_{-m}$ is birational for all $m\geq 3$.
	\end{enumerate}
	
\end{thm}

The example of a degree $1$ del Pezzo surface shows that the bounds given in the theorem are optimal, in particular we have $f(2)=2$ and $b(2)=3$. Borrowing some tools from \cite{OguisoPeternell1995} and applying \cite[Corollary 2]{Reider1988} and a recent result of Jiang \cite[Theorem 1.7]{Jiang2016}, we generalize this theorem to weak Fano varieties of dimension at most $4$.

\begin{thm}\label{Three}
	Let $X$ be a weak Fano threefold with at worst Gorenstein canonical singularities. Then
	\begin{enumerate}		
		\item the linear system $\vert-mK_X\vert$ is basepoint free for $m\geq 2$;
		
		\item the morphism $\Phi_{-m}$ is birational for $m\geq 3$.
	\end{enumerate}
\end{thm}

The lower bounds given in the theorem are both optimal, i.e., $f(3)=2$ and $b(3)=3$. To see this, we consider the following two examples : $X=S_1\times \bbP^1$, where $S_1$ is a del Pezzo surface of degree $1$ and a general hypersurface of degree $6$ in the weighted projective space $\bbP(1,1,1,1,3)$. One can also derive the basepoint freeness of $\vert-2K_X\vert$ by the classification of Gorenstein canonical Fano threefolds with non-empty $\bs\vert-K_X\vert$ given in \cite[Theorem 1.1]{JahnkeRadloff2006}. Moreover, the Fano threefolds with canonical Gorenstein singularities such that $\vert-K_X\vert$ is basepoint free, but not very ample, are called \emph{hyperelliptic}, and they have been classified by Przyjalkowski-Cheltsov-Shramov in \cite[Theorem 1.5]{PrzhiyalkovskiuiChelprimetsovShramov2005}.

\begin{thm}\label{Four}
	Let $X$ be a weak Fano fourfold with at worst Gorenstein canonical singularities. Then
	\begin{enumerate}		
		\item the linear system $\vert-mK_X\vert$ is basepoint free for $m\geq 7$;
		
		\item the rational map $\Phi_{-m}$ is birational for $m\geq 5$.
	\end{enumerate}
\end{thm}

A general hypersurface $X$ of degree $10$ in the weighted projective space $\bbP(1,1,1,1,2,5)$ guarantees that the estimate given by Theorem \ref{Four} (2) is best, i.e., $b(4)=5$. We remark that the variety $X$ is actually a smooth Fano fourfold with index $1$, and the base locus of $\vert-K_X\vert$ is zero dimensional and of degree $10$, given by an equation of the type $x_5^2+x_4^5=0$ in $\bbP(2,5)$ (see \cite[Theorem 6.22]{Heuberger2016}). Combining the argument in low dimension with \cite[Theorem 1.1]{Floris2013} derives a similar result for $n$-dimensional Fano manifolds with index $n-3$.

\begin{thm}\label{coindexfour}
	Let $X$ be a $n$-dimensional Fano manifold with index $n-3$ and let $H$ be the fundamental divisor. Then
	\begin{enumerate}				
		\item the linear system $\vert mH\vert$ is basepoint free for $m\geq 7$;
		
		\item the rational map $\Phi_{\vert mH\vert}$ is birational for $m\geq 5$.
	\end{enumerate}
	
\end{thm}

As in dimension $4$, the estimate in Theorem \ref{coindexfour} (2) is optimal as showed by a general hypersurface of degree $10$ in the weighted projective space $\bbP(1,\cdots,1,2,5)$. The statement Theorem \ref{coindexfour} (1) is weaker than the estimate expected by Fujita's conjecture, which asserts basepoint freeness for $m\geq 4$. However, at the present times, observe that the classification of Fano $n$-folds of index $n-3$ is very far from being known even in dimension four, so our result above might be a starting point of the classification: one may try to describe the Fano $n$-folds of index $n-3$ such that $\Phi_{-m}$ is not birational for all $m\leq 4$.

\section{Preliminaries}

Throughout we work over the field $\bbC$ of complex numbers. Let $X$ be a $n$-dimensional projective normal variety and let $X_0$ be the regular part of $X$ with inclusion map $i\colon X_0\rightarrow X$. Let $\omega_{X_0}=\Omega_{X_0}^n$ be the sheaf of regular $n$-forms over $X_0$. The \emph{canonical divisor} $K_X$ is a Weil divisor on $X$ such that
\[\cO_X(K_X)\vert_{X_0}\cong\omega_{X_0}.\]
A normal projective variety $X$ is said to be \emph{$\bbQ$-Gorenstein}, if some multiple $mK_X$ is a Cartier divisor. If $m=1$ and $X$ is Cohen-Macaulay, then the projective variety $X$ is called \emph{Gorenstein}. Let $\mu\colon X'\rightarrow X$ be a proper birational morphism from a normal projective variety $X'$. If $\Delta\subset X$ is a $\bbQ$ divisor, we denote by $\mu^{-1}_*(\Delta)$ its strict transform. A \emph{log-pair} is a tuple $(X,\Delta)$, where $X$ is a normal projective variety and $\Delta=\sum d_i \Delta_i$ is a $\bbQ$-divisor on $X$ with $0\leq d_i\leq 1$ for all $i$ such that $-(K_X+\Delta)$ is a $\bbQ$-Cartier divisor. If $\Delta=0$, we will abbreviate the log-pair $(X,0)$ as $X$. We will use the standard terminologies for the singularities of pairs given in \cite[\S 2.3]{KollarMori1998}.

A \emph{weak $\bbQ$-Fano variety} $X$ is a $n$-dimensional $\bbQ$-Gorenstein projective variety with nef and big anti-canonical divisor $-K_X$. A weak $\bbQ$-Fano variety $X$ is called $\bbQ$-\emph{Fano}, if the anti-canonical divisor $-K_X$ is ample. The \emph{index} of a Fano variety $X$ is defined as
\[i_X=\sup\{t\in\bbQ\ \vert\ -K_X\sim_{\bbQ}tH, H\ \textrm{is ample and Cartier}\}.\]
The \emph{coindex} of $X$ is defined as $n+1-i_X$. If $X$ has at worst log terminal singularities, then the Picard group $\pic(X)$ is finitely generated and torsion free, so the Cartier divisor $H$ such that $-K_X\sim_{\bbQ}i_X H$ is determined up to linear equivalence and we call it the \emph{fundamental divisor} of $X$.

Let $(X,H)$ be a polarized normal $\bbQ$-factorial projective variety of dimension $n$. Given a torsion-free coherent sheaf $\cE$ on $X$ of positive rank $\rk(\cE)$, the \emph{slope} of $\cE$ with respect to $H$ is defined as
\[\mu_H(\cE)=\frac{c_1(\cE)\cdot H^{n-1}}{\rk(\cE)}.\]
Moreover, if $X$ is smooth, the \emph{discriminant} of $\cE$ by definition is the characteristic class
\[\Delta(\cE)\colon=2\rk(\cE)c_2(\cE)-(\rk(\cE)-1)c_1^2(\cE).\]
\begin{defn}
	Let $(X,H)$ be a polarized normal $\bbQ$-factorial projective variety and let $\cE$ be a torsion-free coherent sheaf. Then $\cE$ is called $H$-stable (resp. $H$-semi-stable) if for any coherent subsheaf $\cF$ of $\cE$ of rank $0<\rk(\cF)<\rk(\cE)$, we have 
	\[\mu_H(\cF)<\mu_H(\cE)\quad (resp.\ \ \mu_H(\cF)\leq \mu_{H}(\cE)).\]
\end{defn}

For a torsion-free coherent sheaf $\cE$ over a polarized projective manifold $(X,H)$, there exists a unique filtration, the so called \emph{Harder-Narasimhan filtration}, by coherent subsheaves
\[0=\cE_0\subsetneq \cE_1\subsetneq\cdots\subsetneq\cE_{s-1}\subsetneq\cE_{s}=\cE\]
such that all the factors $G_i=\cE_i/\cE_{i-1}$ for $i=1,\cdots, s$ are semi-stable sheaves. The sheaf $\cE_1$ is called the \emph{maximal $H$-destabilising subsheaf} of $\cE$. We furthermore introduce the following numbers attached to $\cE$
\[\mu_{H}^{max}(\cE)=\mu_{H}(\cE_1)\ \ \ \textrm{and}\ \ \ \mu_{H}^{min}(\cE)=\mu_{H}(\cE/\cE_{s-1}).\]

The following theorem reduces the study of varieties with only canonical singularities to the study of varieties with only $\bbQ$-factorial terminal singularities.

\begin{thm}\cite[Corollary 1.4.4]{BirkarCasciniHaconMcKernan2010}\label{terminal modification}
	Let $X$ be a normal projective variety with only canonical singularities. Then there is a birational morphism $\mu\colon Y\rightarrow X$, where $Y$ has only $\bbQ$-factorial terminal singularities such that $K_Y=\mu^*K_X$. Such a variety $Y$ is called a terminal modification of $X$.
\end{thm}

Using this existence theorem, it is easy to yield the following result from the terminal setting.
\begin{thm}\cite[Corollary 2]{Reider1988}\cite[Theorem 1.7]{Jiang2016}\label{Trivial Case}
	Let $X$ be a $n$-dimensional normal projective variety with $K_X\sim 0$. Assume that $X$ has at worst canonical singularities and $L$ is a nef and big Cartier divisor on $X$. Then
	\begin{enumerate}
		\item the rational map $\Phi_{\vert mL\vert}$ is birational for $m\geq 3$ if $n=2$;
		
		\item the rational map $\Phi_{\vert mL\vert}$ is birational for $m\geq 5$ if $n=3$.
	\end{enumerate}
\end{thm}

We recall an easy lemma which is nevertheless the key ingredient of our inductive approach to generalize Theorem \ref{Two} to higher dimension. In \cite{Oguiso1991}, it was proved for projective manifolds, but the proof still works for normal projective varieties.

\begin{lemma}\cite[Lemma 1.3]{Oguiso1991}\label{Inductive Lemma}
	Let $X$ be a normal projective variety. Consider an effective Cartier divisor $E$ and an irreducible reduced Cartier divisor $F$ such that $\dim\vert F\vert \geq 1$. Suppose that the restriction $\vert E+F\vert_D$ gives a birational map for a general element $D$ in $\vert F\vert$. Then $\vert E+F\vert$ gives a birational map on $X$.
\end{lemma}

\section{Second Chern class of Fano manifolds}

Let $X$ be a polarized normal projective variety and let $\Chow{X}$ be the Chow space of $X$. The family of rational curves on $X$ is defined as the following subset of $\Chow{X}$,
\begin{align*}
\textrm{RatCurves}(X)\colon=\{ & [Z]\in\Chow{X}\ \lvert\ [Z]\textrm{\ is\ an\ irreducible\ rational\ curve}\}.
\end{align*}
It is well-known that $\textrm{RatCurves(X)}$ is a union of quasi-projective sets. We denote by $\RC{X}$ the normalization of $\textrm{RatCurves}(X)$. For more details we refer to \cite{Kollar1996}. 

Let $\cV$ be an irreducible component of $\RC{X}$. $\cV$ is said to be \emph{a covering family of rational curves} on $X$ if the corresponding universal family dominates $X$. A covering family $\cV$ of rational curves on $X$ is called \emph{minimal} if its general members have minimal degree with respect to some polarization. If $X$ is a uniruled projective manifold, then $X$ carries a minimal covering family of rational curves. We fix such a family $\cV$, and let $[C]\in\cV$ be a general point. Then the tangent bundle $T_X$ can be decomposed on the normalization of $C$ as
\[\cO_{\bbP^1}(2)\oplus\cO_{\bbP^1}(1)^{\oplus d}\oplus\cO_{\bbP^1}^{\oplus (n-d-1)},\]
where $d+2=\det(T_X)\cdot C\geq 2$ is the \emph{anti-canonical degree} of $\cV$.

\begin{defn}
	Let $X$ be a projective manifold. A foliation on $X$ is a non-zero coherent subsheaf $\cF\subsetneq T_X$ such that
	\begin{enumerate}
		\item $\cF$ is closed under the Lie bracket,
		
		\item $\cF$ is saturated in $T_X$ (i.e., $T_X/\cF$ is torsion-free), in particular, $\cF$ is a reflexive sheaf.
	\end{enumerate}
\end{defn}

For any coherent sheaf $\cF$ over $X$, we denote by $\cF^*$ the dual sheaf $\cHom{\cF}{\cO_X}$. Moreover, given a torsion-free $\cF$ of positive rank $r$ on a projective manifold $X$, we denote by $\det(\cF)$ the invertible sheaf $(\wedge^r\cF)^{**}$. If $\cF$ is a foliation, the \emph{canonical class} $K_{\cF}$ of $\cF$ is any Weil divisor on $X$ such that $\cO_X(-K_{\cF})\cong\det(\cF)$. The \emph{Fano foliations} on $X$ are these foliations $\cF\subset T_X$ whose anti-canonical class $-K_{\cF}$ is ample. Similar to Fano manifolds, we can define the index $i_{\cF}$ of a Fano foliation $\cF$ to be the largest positive integer dividing $-K_{\cF}$ in $\pic(X)$ (cf. \cite{AraujoDruel2013}). The following result due to Hwang relates the unstability of $T_X$ to certain special foliation on $X$, which gives also some restrictions of the projective geometry of certain special subvarieties in $\bbP(T_X)$.

\begin{prop}\cite[Proposition 1 and 2]{Hwang1998}\label{Nontangent}
	Let $X$ be a $n$-dimensional Fano manifold with $\rho(X)=1$. Let $H$ be the ample generator of $\pic(X)$, and $\cV$ a minimal covering family of rational curves over $X$. If $T_X$ is not semi-stable, then the maximal $H$-destabilizing subsheaf $\cF$ of $T_X$ defines a foliation on $X$, and general curves in $\cV$ are not tangent to $\cF$.
\end{prop}

Theorem \ref{Chern class} is a consequence of Proposition \ref{Nontangent} and the strenghtening of Bogomolov inequality due to Langer (see \cite[Theorem 5.1]{Langer2004}).

\begin{proof}[Proof of Theorem \ref{Chern class}]
	Without loss of generality, we may assume that $T_X$ is not semi-stable, in particular, we will assume that $i_X\leq n-2$ (cf. \cite{PeternellWisniewski1995,Hwang2001}). Let $\cF$ be the maximal $H$-destabilizing subsheaf of $T_X$. Then $\cF$ defines a Fano foliation on $X$. We denote by $i_{\cF}$ $(>0)$ the index of $\cF$, i.e., $-K_{\cF}=i_{\cF}H$. Fix a minimal covering family $\cV$ of rational curves on $X$ and write
	\[T_X\vert_C=\cO_{\bbP^1}(2)\oplus\cO_{\bbP^1}(1)^{\oplus p}\oplus\cO_{\bbP^1}^{\oplus(n-p-1)}\] for a general point $[C]\in\cV$. Since $T_X/\cF$ is torsion-free, we may assume that the sheaf $\cF$ is a subbundle of $T_X$ along $C$ (cf. \cite[II, Proposition 3.7]{Kollar1996}). Then the restriction of $\cF$ over $C$ is of the following form
	\[\cF\vert_C=\cO_{\bbP^1}(a_1)\oplus\cdots\oplus\cO_{\bbP^1}(a_r),\ a_1\geq\cdots\geq a_r.\]
	Since $\cF\subset T_X$, it follows $a_1\leq 2$, $a_i\leq 1$ for $2\leq i\leq p+1$, and $a_j\leq 0$ for $p+2\leq j\leq r$. However, note that $C$ is not tangent to $\cF$ (cf.~Proposition \ref{Nontangent}), we have actually $a_1\leq 1$. As a consequence, we obtain $0<i_{\cF}<i_X$ and there exists some $1\leq d\leq r$ such that 
	\[a_1=\cdots=a_d=1>a_{d+1}\geq\cdots\geq a_r.\] 
	
	On the other hand, if $\rk(\cF)=r=d$, then $\cF\vert_C$ is ample. As $\rho(X)=1$, this implies $X\cong \bbP^n$ \cite[Proposition 2.7]{AraujoDruelKovacs2008} (see also \cite[Corollary 2.3]{Liu2017}). This contradicts our assumption that the tangent bundle $T_X$ is not semi-stable. So we have $d<r$. Moreover, by definition we have
	\[c_1(\cF)\cdot C=i_{\cF}H\cdot C=\sum_{i=1}^r a_i\leq d.\]
	Since $\cF\vert_C$ is a subbundle of $T_X\vert_C$, it follows that the positive part 
	\[(\cF\vert_C)^+\colon=\cO_{\bbP^1}(a_1)\oplus\cdots\oplus\cO_{\bbP^1}(a_d)\cong \cO_{\bbP^1}(1)\oplus\cdots\oplus\cO_{\bbP^1}(1)\]
	of $\cF\vert_C$ is a subbundle of $T_X\vert_C$. In particular, the ample vector bundle $(\cF\vert_C)^+$ is also a subbundle of the positive part 
	\[(T_X\vert_C)^{+}\colon =\cO_{\bbP^1}(2)\oplus\cO_{\bbP^1}(1)^{\oplus p}\]
	of $T_X\vert_C$. This implies that $d\leq p+1$. However, if $d=p+1$, then $(\cF\vert_C)^+$ and $(T_X\vert_C)^{+}$ have the same rank, and $(\cF\vert_C)^+$ is a subbundle of $(T_X\vert_C)^{+}$ if and only if $(\cF\vert_C)^+=(T_X\vert_C)^{+}$. This is impossible. Therefore, we have $d\leq p$. In summary, we have
	\begin{equation}\label{variousrelation}
	1\leq i_{\cF}\leq i_X-1,\ i_{\cF}H\cdot C\leq d\leq p=i_X H\cdot C-2,\ 1\leq d\leq r-1.
	\end{equation}
	
	Furthermore, since $X$ is a Fano manifold, its tangent bundle $T_X$ is generically ample (cf. \cite[Theorem 1.3]{Peternell2012}), and it follows $\mu_{H}^{\min}(T_X)>0$. We denote by
	\[0=\cE_0\subsetneq \cE_1\subsetneq\cdots\subsetneq\cE_{s-1}\subsetneq \cE_s=T_X\]
	the Harder-Narasimhan filtration of $T_X$. By our assumption, we have $s\geq 2$, $\cE_1=\cF$ and 
	\[\mu^{min}_H(T_X)=\mu_{H}(T_X/\cE_{s-1}).\]
	As $\rho(X)=1$ and $\mu_{H}^{min}(T_X)>0$, we have $c_1(T_X/\cE_{s-1})=i H$ for some $i\in\bbZ_{> 0}$. Moreover, note that $\rk(T_X/\cE_{s-1})\leq n-1$, it follows
	\[\mu_{H}^{min}(T_X)=\mu_{H}(T_X/\cE_{s-1})=\frac{i}{\rk(T_X/\cE_{s-1})}H^n\geq \frac{1}{n-1}H^n.\]
	
	\textit{Proof of (1).}
	In view of \eqref{variousrelation}, we have $i_{\cF}=1$. Moreover, we have
	\[0<H\cdot C\leq d\leq p=2H\cdot C-2.\]
	This implies $H\cdot C\geq 2$. Then we have
	\[\mu_H^{\max}(T_X)=\mu_H(\cF)=\frac{1}{r}H^n\leq\frac{1}{d+1}H^n\leq \frac{1}{H\cdot C+1}H^n\leq \frac{1}{3}H^n.\]
	Now the strenghtening of Bogomolov inequality (see \cite[Theorem 5.1]{Langer2004})
	\[H^n\cdot \left(\Delta(T_X)\cdot H^{n-2}\right)+n^2\left(\mu_{H}^{\max}(T_X)-\mu_{H}(T_X)\right)(\mu_{H}(T_X)-\mu_{H}^{\min}(T_X))\geq 0,\]
	implies
	\[H^n\cdot\left(\Delta(T_X)\cdot H^{n-2}\right)+n^2\left(\frac{1}{3}-\frac{2}{n}\right) H^n\cdot \left(\frac{2}{n}-\frac{1}{n-1}\right) H^n\geq 0.\]
	By the definition of $\Delta(T_X)$, after simplifying the expression, we conclude
	\[c_2(X)\cdot H^{n-2}\geq \frac{11n-16}{6n-6}H^n.\]
	
	\textit{Proof of (2).} According to \eqref{variousrelation}, we have $i_{\cF}\leq i_X-1$. Similar to the proof of (1), we will establish the following upper bound for $\mu_{H}^{max}(T_X)$ in terms of the index of $X$
	\[\mu_{H}^{max}(T_X)\leq\frac{i_X-2}{i_X-1}H^n.\]
	If $1\leq i_{\cF}\leq i_X-2$, by \eqref{variousrelation}, one can derive
	\[\mu_{H}^{\max}(T_X)=\mu_{H}(\cF)=\frac{i_{\cF}}{r}H^n\leq \frac{i_{\cF}}{d+1}H^n\leq \frac{i_{\cF}}{i_{\cF}H\cdot C+1}H^n.\]
	Note that $H\cdot C\geq 1$ and $i_{\cF}\leq i_X-2$, we conclude
	\[\mu_{H}^{\max}(T_X)\leq \frac{i_{\cF}}{i_{\cF}+1}H^n\leq \frac{i_X-2}{i_X-1}H^n.\]
	
	If $i_{\cF}=i_X-1$, the inequality given in \eqref{variousrelation}
	\[0<i_{\cF}H\cdot C=(i_X-1)H\cdot C\leq p= i_X H\cdot C-2,\]
	yields $H\cdot C\geq 2$. Nevertheless, as $i_X\geq 3$, it follows
	\[\mu_{H}(\cF)\leq\frac{i_{\cF}}{d+1}H^n\leq\frac{i_X-1}{(i_X-1) H\cdot C+1}H^n \leq \frac{i_X-2}{i_X-1}H^n.\]
	Now, by the strenghtening of Bogomolov inequality again, it follows
	\[H^n\cdot\left(\Delta(T_X)\cdot H^{n-2}\right)+n^2\left(\frac{i_X-2}{i_X-1}-\frac{i_X}{n}\right)H^n\cdot\left(\frac{i_X}{n}-\frac{1}{n-1}\right)H^n\geq 0.\]
	After simplifying the expression, one can derive the following inequality 
	\[c_2(X)\cdot H^{n-2}\geq \frac{i_X(i_X-1)}{2} H^n+\frac{2n(i_X-1)-i_X^2}{2(n-1)(i_X-1)}H^n.\]
	Then we can conclude by observing that the number $2n(i_X-1)-i_X^2$ is positive for $3\leq i_X\leq n-2$ and $n\geq 7$.
\end{proof}

\begin{remark}
	If $n\geq 7$, one can easily check that the upper bound of $\mu_H^{\max}(T_X)$ given in the proof is strictly smaller than $\mu_H(T_X)$ if $i_X\geq n-1$. This actually recovers the semi-stability of $T_X$ in the case $i_X\geq n-1$ and $n\geq 7$.
\end{remark}

\section{Anti-canonical geometry of Fano varieties}

This section is devoted to study Question \ref{Question} in several special cases. The crucial ingredient is the existence of good ladder on Fano varieties with mild singularities. The following reduction result suggested by the referee is obtained as an application of the Basepoint Free Theorem.

\begin{prop}\label{WeakFano to Fano}
	In Question \ref{Question}, one may assume that $X$ is a Fano variety with at worst Gorenstein canonical singularities.
\end{prop}

\begin{proof}
	Since $-K_X$ is a nef and big Cartier divisor, by the Basepoint Free Theorem, there exists a projective birational morphism $\phi\colon X\rightarrow X'$ to a normal projective variety $X'$ with connected fibers such that $-kK_{X}\sim \phi^*A$ for some ample divisor $A$ on $X'$ and some positive integer $k$. It follows that $-kK_{X^\prime}\sim A$ as $\phi_*K_X=K_{X^\prime}$. In particular, $K_{X^{\prime}}$ is Cartier and $\phi^*K_{X^{\prime}}=K_X$. This implies that the variety $X'$ has at worst canonical Gorenstein singularities. On the other hand, we have $h^0(X,-mK_X)=h^0(X^\prime,-mK_{X^{\prime}})$ for any $m\in\bbZ$ as $\phi_*\cO_X=\cO_{X'}$. In particular, the pluri-anti-canonical maps $\Phi_{\vert-mK_X\vert}$ of $X$ factor as $\phi$ followed by the pluri-anti-canonical maps $\Phi_{\vert-mK_{X^{\prime}}\vert}$ of $X^\prime$. As a consequence, the map $\Phi_{\vert-mK_X\vert}$ is birational if and only if $\Phi_{\vert-mK_{X^{\prime}}\vert}$ is birational, and the complete linear system $\vert-mK_X\vert$ is basepoint free if and only if the complete linear system $\vert-mK_{X'}\vert$ is basepoint free.
	\qed
\end{proof}

\subsection{Fano threefolds and Fano fourfolds} Theorem \ref{Three} is an improvement of a result of Fukuda \cite[Main Theorem]{Fukuda1991} and a result of Chen and Jiang \cite[Corollary 5.13]{ChenJiang2016}. We shall prove it as a consequence of Lemma \ref{Inductive Lemma} and \cite[Main Theorem]{Ambro1999}.

\begin{proof}[Proof of Theorem \ref{Three}]
	Thanks to Proposition \ref{WeakFano to Fano}, we may assume that $X$ is a Fano threefold with at worst Gorenstein canonical singularities of dimension $3$. Let $S\in\vert -K_X\vert$ be a general member. Then the pair $(X,S)$ is plt by \cite[Main Theorem]{Ambro1999}. According to the inversion of adjunction (cf. \cite[Theorem 5.50]{KollarMori1998}), it follows that $(S,0)$ has at worst klt singularties. Note that $S$ is Gorenstein, so $S$ has at worst Gorenstein canonical singularities.
	
	\textit{Step 1. Basepoint freeness of $\vert-mK_X\vert$ for $m\geq 2$.} Since $-K_X\vert_S$ is nef and big, $K_S\sim 0$ and $S$ has at worst canonical singularities, the complete linear system $\vert-mK_X\vert_S\vert$ is basepoint free for $m\geq 2$ by \cite[Theorem 3.1]{Kawamata2000}. Consider the following exact sequence
	\[0\rightarrow \cO_X(-(m-1)K_X)\rightarrow\cO_X(-mK_X)\rightarrow\cO_S(-mK_X\vert_S)\rightarrow 0.\]	
	Note that we have $H^1(X,-mK_X)=0$ for $m\in\bbZ$ by Kawamata-Viehweg vanishing theorem and Serre duality. This implies that the natural restriction map
	\[H^0(X,-mK_X)\rightarrow H^0(S,-mK_X\vert_S)\]
	is surjective for $m\in\bbZ$. In particular, we have $\bs\vert-mK_X\vert=\bs\vert-mK_X\vert_S\vert$ for $m\in\bbZ$. As a consequence, the linear system $\vert-mK_X\vert$ is basepoint free for $m\geq 2$.
	
	\textit{Step 2. The morphism given by $\vert-mK_X\vert$ is birational for $m\geq 3$.} Combining the fact $H^1(X,\cO_X)=0$ and the following natural exact sequence of sheaves
	\[0\rightarrow \cO_X\rightarrow \cO_X(-K_X)\rightarrow \cO_S(-K_X\vert_S)\rightarrow 0\]
	gives $h^0(X,-K_X)=h^0(S,-K_X\vert_S)+1$. Moreover, by \cite[Theorem 3.1]{Kawamata2000}, we have $h^0(S,-K_X\vert_S)\geq 1$. Therefore, we obtain $\dim\vert -K_X\vert\geq 1$. On the other hand, note that the rational map defined by $\vert-mK_X\vert_S\vert$ coincides with the restriction of the rational map given by $\vert-mK_X\vert$ over $S$ for $m\geq 0$. Thanks to Theorem \ref{Trivial Case} (1), the rational map given by $\vert-mK_X\vert_S\vert$ is birational for $m\geq 3$. Then Lemma \ref{Inductive Lemma} implies that the linear system $\vert-mK_X\vert$ gives a birational map for $m\geq 3$. 
\end{proof}

\begin{remark}
	As mentioned in the introduction, one can also derive the same result from the classification given in \cite{PrzhiyalkovskiuiChelprimetsovShramov2005,JahnkeRadloff2006}. The advantage of our argument is that it may be applied in higher dimension to get some upper bound of $f(n)$, the disadvantage is that we do not get any information about the explicit geometry of $X$ or any information about the base locus $\bs\vert-K_X\vert$. 
\end{remark}

\begin{proof}[Proof of Theorem \ref{Four}]
	By Proposition \ref{WeakFano to Fano}, we may assume that $X$ is a Fano fourfold with at worst Gorenstein canonical singularities. By \cite[Proposition 3.2]{Floris2013}, we obtain that $h^0(X,-K_X)\geq 2$. Let $Y\in\vert -K_X\vert$ be a general member. Then $Y$ has only Gorenstein canonical singularities by \cite[Theorem 5.2]{Kawamata2000}. Owe to \cite[Theorem 2]{OguisoPeternell1995}, the linear system $\vert -mK_X\vert_{Y}\vert$ is basepoint free for $m\geq 7$. It follows that the linear system $\vert -mK_X\vert$ is basepoint free for $m\geq 7$ as $\bs\vert-mK_X\vert=\bs\vert -mK_X\vert_Y\vert$ by Kawamata-Viehweg vanishing theorem. Furthermore, applying Theorem \ref{Trivial Case} (2) and the same argument in the proof of Theorem \ref{Three}, one can conclude that the complete linear system $\vert-mK_X\vert$ gives a birational map for $m\geq 5$. 
\end{proof}

\begin{remark}
	Let $X$ be a $n$-dimensional normal projective variety with at worst canonical singularities such that $K_X\sim 0$. Let $L$ be an arbitrary nef and big line bundle over $X$. Then the complete linear system $\vert mL\vert$ is basepoint free for $m\geq f(n+1)$ if $n=2$ (cf. Theorem \ref{Three} and \cite[Theorem 3.1]{Kawamata2000}), and the morphism $\Phi_{\vert mL\vert}$ corresponding to $\vert mL\vert$ is birational for $m\geq b(n+1)$ if $n=2$ or $3$ (cf. Theorem \ref{Three}, Theorem \ref{Four} and Theorem \ref{Trivial Case}). So it is natural and interesting to ask if these results still hold in higher dimension.
\end{remark}

\subsection{Fano manifolds of coindex four} It is well-known that the index $i_X$ of a Fano manifold $X$ of dimension $n$ does not exceed $n+1$. By Kobayashi-Ochiai theorem \cite{KobayashiOchiai1973}, if $i_X=n+1$, then $X\cong \bbP^n$ and if $i_X=n$, then $X\cong Q^n\subset\bbP^{n+1}$ is a quadric hypersurface. Fano manifolds of index $n-1$ (del Pezzo manifolds) have been classified by Fujita and Fano manifolds of index $n-2$ (Mukai manifolds) were classified by Mukai under the existence of good ladder which was confirmed later by Mella \cite{Mella1999}. 

\begin{proof}[Proof of Theorem \ref{Non-vanishingH}]
	By \cite[Theorem 1.2]{Floris2013}, we may assume that $X$ is a Fano manifold of dimension $6$ or $7$. If $\rho(X)\geq 2$ and $n=7$, according to the explicit classification results given in \cite{Wisniewski1991a}, one can easily check it.
	
	Next we consider the case $\rho(X)\geq 2$ and $n=6$. By \cite[Corollary]{Wisniewski1993}, in this case, $\rho(X)\geq 3$ if and only if $X\cong\bbP^2\times\bbP^2\times\bbP^2$ and we are done. The Fano manifolds $X$ of dimension $2r$ and index $r$ such that $r\geq 3$ and $\rho(X)=2$ are classified by Wi{\'s}niewski in \cite{Wisniewski1994}. In particular, according to \cite[Proposition 3.2 and 3.3]{Wisniewski1994}, there exists an extremal contraction $p\colon X\rightarrow Y$ to a Fano manifold $Y$ such that one of the following occurs:
	\begin{enumerate}
		\item $\dim(Y)=r+1$ and the direct image sheaf $\cE\colon=p_*\cO_X(H)$ is a locally free sheaf of rank $r$ such that $X\cong\bbP(\cE)$;
		
		\item $\dim(Y)=r$ and the direct image sheaf $\cE\colon=p_*\cO_X(H)$ is a locally free sheaf of rank $r+2$ and $X$ is a quadric bundle over $Y$;
		
		\item $\dim(Y)=r+1$ and the direct image sheaf $\cE\colon=p_*\cO_X(H)$ is a reflexive non-locally free sheaf of rank $r$ and $X\cong\bbP(\cE)$.
	\end{enumerate}
	In particular, we have $h^0(X,H)=h^0(Y,\cE)$. In Case (1), $\cE$ is the direct sum of ample line bundles over $Y$ except $Y\cong Q^4$ and $\cE\cong \textbf{E}(1)\oplus \cO(1)$, where $\textbf{E}$ is the spinor bundle over $Q^4$ (cf. \cite[(4.1)]{Wisniewski1994}). In the former case, it is easy to see that we have $h^0(Y,\cE)\geq 4$ if $r=3$. In the latter case, we can conclude by the fact $h^0(Q^4,\textbf{E}(1))=4$ (cf. \cite[Theorem 2.3]{Ottaviani1988}). In Case (2), $\cE$ is a globally generated vector bundle of rank $r+2$ (cf. \cite[(5.1)]{Wisniewski1994}), so we have $h^0(Y,\cE)\geq r+2$ and we are done. In Case (3), there exists an extension $\cF$ of $\cE$ by $\cO_Y$ (cf. \cite[(6.2)]{Wisniewski1994})
	\[0\rightarrow \cO_Y\rightarrow \cF\rightarrow \cE\rightarrow 0.\]
	In particular, we have $h^0(Y,\cE)=h^0(Y,\cF)-1$ since $H^1(Y,\cO_Y)=0$. Moreover, the sheaf $\cF$ is an ample vector bundle of rank $r$ except No. $2$ in \cite[(6.3)]{Wisniewski1994}. If $\cF$ is ample, it is easy to check $h^0(Y,\cF)\geq 5$ if $r=3$. For No. $2$ in \cite[(6.3)]{Wisniewski1994}, $X$ is actually a divisor of bidegree $(1,2)$ in $\bbP^{r+1}\times\bbP^r$ (cf. \cite[Theorem 6.8]{BallicoWisniewski1996}), then we can conclude by a straightforward computation.
	
	Now we consider the case $\rho(X)=1$. Note that in this case $T_X$ is semi-stable if $n=6$ (cf. \cite[Theorem 3]{Hwang1998}), so our result follows from \cite[Proposition 3.3]{Floris2013} in this case. It remains to consider the case $n=7$ and $\rho(X)=1$. By Theorem \ref{Chern class} (2), we get 
	\[c_2(X)\cdot H^5\geq 6H^7.\]
	In view of the proof of \cite[Proposition 3.3]{Floris2013}, we have
	\begin{align*}
	h^0(X,H) & = -\frac{1}{3}H^7+\frac{1}{12} c_2(X)\cdot H^5+4\geq \frac{1}{6}H^7+4>4.
	\end{align*}
	This completes the proof.
\end{proof}

\begin{remark}
	Combining Theorem \ref{Non-vanishingH} with \cite[Theorem 1.1]{Floris2013} gives the existence of ladder with at worst canonical singularities over $n$-dimensional Fano manifolds of index $n-3$. However, due to an example of H\"oring and Voisin \cite[Example 2.12]{HoeringVoisin2011}, one can not expect that there exists a smooth ladder over such a Fano manifold.
\end{remark}

\begin{proof}[Proof of Theorem \ref{coindexfour}]
	By Theorem \ref{Non-vanishingH} and \cite[Theorem 1.1]{Floris2013}, there exists a good ladder. More precisely, there is a descending sequence of subvarieties
	\[X=Y_0\supsetneq Y_1\supsetneq Y_2\supsetneq\cdots\supsetneq Y_{n-4}\supsetneq Y_{n-3}\]
	such that the variety $Y_{i+1}\in\vert H\vert_{Y_i}\vert$ has at most Gorenstein canonical singularities and $K_{Y_{n-3}}\sim 0$. 
	
	Thanks to \cite[Theorem 2]{OguisoPeternell1995}, the complete linear system $\left\vert mH\vert_{Y_{n-3}}\right\vert$ is basepoint free when $m\geq 7$. By Kawamata-Viehweg vanishing theorem and Serre duality, the natural restriction 
	\[H^0(X,mH)\rightarrow H^0(Y_{n-3},mH\vert_{Y_{n-3}})\]
	is surjective for $m\in\bbZ$, so the linear system $\vert mH\vert$ is basepoint free for $m\geq 7$.
	
	Note that the rational map defined by $\vert mH\vert_{Y_{n-3}}\vert$ is birational for $m\geq 5$ (cf.~Theorem \ref{Trivial Case}). By the same argument that we used in the proof of Theorem \ref{Three}, we conclude that the rational map given by $\vert mH\vert_{Y_{n-4}}\vert$ is birational for $m\geq 5$. Therefore, after an inductive argument on $n$, we see that the rational map given by $\vert mH\vert$ is birational for $m\geq 5$.
\end{proof}

\begin{remark}
	Let $X$ be a $n$-dimensional weak Fano variety with at worst Gorenstein canonical singularities. Let $H$ be its fundamental divisor. If its index $i_X$ is equal to $n-1$ or $n-2$, then the same argument in the proof of Theorem \ref{coindexfour} together with  Theorem \ref{Two}, Theorem \ref{Three} and Theorem \ref{Trivial Case} shows that the complete linear system $\vert mH\vert$ is basepoint free for $m\geq 2$ and the morphism $\Phi_{\vert mH\vert}$ is birational for $m\geq 3$. The existence of good ladder on such varieties was proved in \cite[Main Theorem]{Ambro1999} (see also \cite[Theorem 5.1]{Mella1999,Kawamata2000}).
\end{remark}

\begin{remark}
	Theorem \ref{Non-vanishingH} and Theorem \ref{coindexfour} are true even for weak Fano fivefolds with at worst canonical Gorenstein singularties. However, in higher dimension, the argument of Theorem \ref{Chern class} relies on the smoothness, so it may be interesting to find an alternative proof of Theorem \ref{Non-vanishingH} which does not depend on the smoothness; then, we expect that Theorem \ref{coindexfour} still holds for weak Fano varieties with at worst canonical Gorenstein singularities.
\end{remark}

\subsection*{Acknowledgements} 	

I heartly thank my advisor Andreas H{\"o}ring for valuable discussions, suggestions and for his careful proofreading of the rather awkward draft of this paper. I want to thank Chen Jiang and St\'ephane Druel for helpful communications and comments. I also would like to thank the anonymous reviewer for pointing out several inacuracies in previous versions and for his valuable comments and suggestions to improve the explanation of the paper. This paper was written while I stayed in Institut de Recherche Math{\'e}matique de Rennes (IRMAR) and I would like to thank the institution for the hospitality and support.

\def\cprime{$'$}

\renewcommand\refname{References}
\bibliographystyle{alpha}
\bibliography{birationality}

\end{document}